\documentclass[letterpaper,12pt]{amsart}
\title{Beyond G\"ollnitz' Theorem I: A Bijective Approach}
\author{Isaac Konan}
\address{IRIF, Universit\'e de Paris, B\^atiment Sophie Germain, Case courrier 7014, 8 Place Aur\'elie Nemours, 75205 Paris Cedex 13, France}
\email{konan@irif.fr}
\date{}
\usepackage{amssymb}
\usepackage{amsmath}
\usepackage{amsthm}
\usepackage[top=1in,bottom=1in,left=1in,right=1in,includehead]{geometry}
\usepackage{enumitem}
\usepackage{xcolor}
\usepackage{graphicx}
\usepackage{pgfarrows,pgfnodes}
\usepackage{url}
\usepackage{textcomp}
\usepackage{ytableau}
\usepackage{dsfont}
\usepackage[affil-it]{authblk}

\definecolor{foge}{rgb}{0.1, 0.6, 0.1}

\newcommand{\So}{\textbf{Step 1 }}
\newcommand{\N}{\mathbb{N}}
\newcommand{\Soo}{\textbf{Step 1}}
\newcommand{\St}{\textbf{Step 2 }}
\newcommand{\Stt}{\textbf{Step 2}}
\newcommand{\Thm}[1]{\textbf{Theorem \ref{#1}}}

\newcommand{\Lem}[1]{\textbf{Lemma \ref{#1}}}
\newcommand{\Prp}[1]{\textbf{Proposition \ref{#1}}}
\newcommand{\Sct}[1]{Section \ref{#1}}
\newcommand{\Od}{\mathcal{O}}
\newcommand{\C}{\mathcal{C}}
\newcommand{\Cc}{\mathcal{C}_\rtimes}

\newcommand{\Pp}{\mathcal{P}}
\newcommand{\Sc}{\mathcal{S}}

\newcommand{\E}{\mathcal{E}}
\newcommand{\Ee}{\mathcal{E}_1}
\newcommand{\Eee}{\mathcal{E}_2}

\newcommand{\la}{\lambda}

\numberwithin{equation}{section}
\newtheorem{theo}{Theorem}[section]
\newtheorem{prop}{Proposition}[section]
\newtheorem{lem}{Lemma}[section]
\newtheorem{cor}{Corollary}[section]

\begin{document}
\maketitle
\begin{abstract}
In 2003, Alladi, Andrews and Berkovich proved an identity for
partitions where parts occur in eleven colors: four primary colors,
six secondary colors, and one quaternary color.    Their work answered
a longstanding question of how to go beyond a classical theorem of
G\"ollnitz, which uses three primary and three secondary colors.
Their main tool was a deep and difficult four parameter $q$-series
identity. In this paper we take a different approach.   Instead of
adding an eleventh quaternary color, we introduce forbidden patterns
and give a bijective proof of a ten-colored partition identity lying
beyond G\"ollnitz' theorem.   Using a second bijection, we show that our
identity is equivalent to the identity of Alladi, Andrews, and
Berkovich. From a combinatorial viewpoint, the use of forbidden
patterns is more natural and leads to a simpler formulation.
In fact, in Part II of this series we will show how our method can be
used to go beyond G\"ollnitz' theorem to any number of primary colors.
\end{abstract}
\ytableausetup{centertableaux,boxsize=0.31cm}
\section{Introduction and Statements of Results}
\subsection{History}
A partition of a positive integer $n$ is a non-increasing sequence of positive integers whose sum is equal to $n$. For example, the partitions of $7$ are 
\[(7),(6,1),(5,2),(5,1,1),(4,3),(4,2,1),(4,1,1,1),(3,3,1),(3,2,2),(3,2,1,1),\]
\[(3,1,1,1,1),(2,2,2,1),(2,2,1,1,1),(2,1,1,1,1,1)\,\,\text{and}\,\,(1,1,1,1,1,1,1)\,\cdot\] 
The study of partition identities has a long history, dating back to
Euler's proof that there are as many partitions of $n$ into distinct parts as partitions of $n$ into odd parts. The corresponding identity is 
\begin{equation}
(-q;q)_\infty = \frac{1}{(q;q^2)_\infty}\,,
\end{equation}
where \[(x;q)_m = \prod_{k=0}^{m-1} (1-xq^k)\,\,,\]
for  any $m\in \mathbb{Z}_{\geq 0}\cup \{\infty\}$ and $x,q$ such that $|q|<1$.
\\\\One of the most important identities in the theory of partitions is
Schur's theorem \cite{Sc26}.
\begin{theo}[Schur 1926]\label{scr}
For any positive integer $n$, the number of partitions of $n$ into distinct parts congruent to $\pm 1 \mod 3$ is equal to the number of partitions of $n$ where parts differ by at least three and multiples of three differ by at least six.
\end{theo}
There have been a number of proofs of Schur's result over the years,
including a $q$-difference equation proof of Andrews \cite{AN68} and a simple
bijective proof of Bressoud \cite{BR80}.
\\\\Another important identity is G\"ollnitz' theorem \cite{GO67}.
\begin{theo}[G\"ollnitz 1967]\label{goll}
For any positive integer $n$, the number of partitions of $n$ into distinct parts  congruent to $2,4,5 \mod 6$ is equal to the number of partitions of $n$ into parts  different from $1$ and $3$, and where parts differ by at least six with equality only if parts are congruent to $2,4,5 \mod 6$.
\end{theo}
Like Schur's theorem, G\"ollnitz's identity can be proved using
$q$-difference equations \cite{AN69} and an elegant Bressoud-style bijection \cite{PRS04,ZJ15}.
\\\\
Seminal work of Alladi, Andrews, and Gordon in the 90's showed how the
theorems of Schur and G\"ollnitz emerge from more general results on
colored partitions \cite{AAG95}. 
\\\\In the case of Schur's theorem, we
consider parts in three colors $\{a,b,ab\}$ and order them as follows:
\begin{equation}
1_{ab}<1_a<1_b<2_{ab}<2_a<2_b<3_{ab}<\cdots\,\cdot
\end{equation}
We then consider the partitions with colored parts different from $1_{ab}$ and satisfying the minimal difference conditions in the table 
\begin{equation}\label{scrrr}
\begin{array}{|c|cc|c|}
\hline
_{\lambda_i}\setminus^{\lambda_{i+1}}&a&b&ab\\
\hline
a&1&2&1\\
b&1&1&1\\
\hline
ab&2&2&2\\
\hline
\end{array}\,\cdot
\end{equation}
Here, the part $\la_i$ with color in the row and the part $\la_{i+1}$ with color in the column differ by at least the corresponding entry in the table.  
An example of such a partition is $(7_{ab},5_{b},4_{a},3_{ab},1_b)$. The Alladi-Gordon refinement of Schur's partition theorem \cite{AG93} is stated as follows:
\begin{theo}\label{ag}
Let $u,v,n$ be non-negative integers. Denote by $A(u,v,n)$ the number of partitions of $n$ into $u$ distncts parts with color $a$ and $v$ distinct parts with color $b$, and denote by $B(u,v,n)$ the number of partitions of $n$ satisfying the conditions in \eqref{scrrr}, with $u$ parts with color $a$ or $ab$, and $v$ parts with color $b$ or $ab$. We then have  $A(u,v,n)=B(u,v,n)$ and the identity
\begin{equation}
\sum_{u,v,n\geq 0} B(u,v,n)a^ub^vq^n = \sum_{u,v,n\geq 0} A(u,v,n)a^ub^vq^n = (-aq;q)_\infty (-bq;q)_\infty\,\cdot
\end{equation}
\end{theo}
Note that a transformation implies Schur's theorem :
\begin{equation}
\left\lbrace 
\begin{array}{l r c l }
\text{dilation :} &q &\mapsto&q^3\\
\text{translations :} &a,b &\mapsto&q^{-2},q^{-1}\\
\end{array}
\right. \,\cdot
\end{equation}
In fact, the minimal difference conditions given in \eqref{scrrr} give after these transformations the minimal differences in Schur's theorem. 
\\\\In the case of G\"ollnitz' theorem,
we consider parts that occur in six colors $\{a,b,c,ab,ab,bc\}$ with the order
\begin{equation}
1_{ab}<1_{ac}<1_a<1_{bc}<1_b<1_c<2_{ab}<2_{ac}<2_a<2_{bc}<2_b<2_c<3_{ab}<\cdots\,,
\end{equation}
and the partitions with colored parts different from $1_{ab},1_{ac},1_{bc}$ and satisfying the minimal difference conditions in 
\begin{equation}\label{eq:aag}
\begin{array}{|c|ccc|ccc|}
\hline
_{\lambda_i}\setminus^{\lambda_{i+1}}&a&b&c&ab&ac&bc\\
\hline
a&1&2&2&1&1&2\\
b&1&1&2&1&1&1\\
c&1&1&1&1&1&1\\
\hline
ab&2&2&2&2&2&2\\
ac&2&2&2&1&2&2\\
bc&1&2&2&1&1&2\\
\hline
\end{array}\,\cdot
\end{equation}
The Alladi-Andrews-Gordon refinement of G\"ollnitz's partition theorem can be stated as follows:
\begin{theo}\label{aag}
Let $u,v,w,n$ be non-negative integers. Denote by $A(u,v,w,n)$ the number of partitions of $n$ into $u$ distncts parts with color $a$, $v$ distinct parts with color $b$ and $w$ distinct parts with color $c$, and denote by $B(u,v,w,n)$ the number of partitions of $n$ satisfying the conditions in \eqref{eq:aag}, with $u$ parts with color $a,ab$ or $ac$, $v$ parts with color $b,ab$ or $bc$ and $w$ parts with color $c,ac$ or $bc$. We then have  $A(u,v,w,n)=B(u,v,w,n)$ and the identity
\begin{equation}
\sum_{u,v,w,n\geq 0} B(u,v,w,n)a^ub^vc^wq^n = \sum_{u,v,w,n\geq 0} A(u,v,w,n)a^ub^vc^wq^n = (-aq;q)_\infty (-bq;q)_\infty(-cq;q)_\infty\,\cdot
\end{equation}
\end{theo}
Note that a transformation implies G\"ollnitz' theorem :
\begin{equation}
\left\lbrace 
\begin{array}{l r c l }
\text{dilation :} &q &\mapsto&q^6\\
\text{translations :} &a,b,c &\mapsto&q^{-4},q^{-2},q^{-1}\\
\end{array}
\right. \,\cdot
\end{equation}
\\Observe that while Schur's theorem is not a direct corollary of
G\"ollnitz' theorem, \Thm{ag} \emph{is} implied by \Thm{aag} by
setting $c= 0$.    Therefore G\"ollnitz' theorem may be viewed as a level
higher than Schur's theorem, since it requires three primary colors
instead of two.
\\\\Following the work of Alladi, Andrews, and Gordon, it was an open
problem to find a partition identity beyond G\"ollnitz' theorem, in the
sense that it would arise from four primary colors.    This was
famously solved by Alladi, Andrews, and Berkovich \cite{AAB03}. To describe their result, 
we consider parts that occur in eleven colors $\{a,b,c,d,ab,ac,ad,bc,bd,cd,abcd\}$ and ordered as follows:
\begin{equation}\label{quat}
1_{abcd}<1_{ab}<1_{ac}<1_{ad}<1_a<1_{bc}<1_{bd}<1_b<1_{cd}<1_c<1_d<2_{abcd}<\cdots\,\cdot
\end{equation}
Let us consider the partitions with the size of the secondary parts greater than one and satisfying the minimal difference conditions in 
\begin{equation}\label{tab2}
\begin{array}{|c|cccc|ccc|cc|c|}
\hline
_{\lambda_i}\setminus^{\lambda_{i+1}}&ab&ac&ad&a&bc&bd&b&cd&c&d\\
\hline
ab&2&2&2&2&2&2&2&2&2&2\\
ac&1&2&2&2&2&2&2&2&2&2\\
ad&1&1&2&2&2&2&2&2&2&2\\
a&1&1&1&1&2&2&2&2&2&2\\
\hline
bc&1&1&1&1&2&2&2&2&2&2\\
bd&1&1&1&1&1&2&2&2&2&2\\
b&1&1&1&1&1&1&1&2&2&2\\
\hline
cd&1&1&1&1&1&1&1&2&2&2\\
c&1&1&1&1&1&1&1&1&1&2\\
\hline
d&1&1&1&1&1&1&1&1&1&1\\
\hline
\end{array}\,,
\end{equation}
and such that parts with color $abcd$ differ by at least $4$, and the smallest part with color $abcd$ is at least equal to  $4+2\tau-\chi(1_a\text{ is a part})$, where $\tau$ is the number of primary and secondary parts in the partition.
The theorem is then stated as follows.
\begin{theo}\label{th2}
Let $u,v,w,t,n$ be non-negative integers. Denote by $A(u,v,w,t,n)$ the number of partitions of $n$ into $u$ distncts parts with color $a$, $v$ distinct parts with color $b$, $w$ distinct parts with color $c$ and $t$ distinct parts with color $d$, and denote by $B(u,v,w,t,n)$ the number of partitions of $n$ satisfying the conditions in \eqref{tab2}, with $u$ parts with color $a,ab,ac,ad$ or $abcd$, $v$ parts with color $b,ab,bc,bd$ or $abcd$, $w$ parts with color $c,ac,bc,cd$ or $abcd$ and $t$ parts with color $d,ad,bd,cd$ or $abcd$.
We then have $A(u,v,w,t,n)=B(u,v,w,t,n)$ and the identity
\begin{equation}
\sum_{u,v,w,t,n\geq 0} B(u,v,w,t,n)a^ub^vc^wd^tq^n =
(-aq;q)_\infty (-bq;q)_\infty(-cq;q)_\infty(-dq;q)_\infty\,\cdot
\end{equation}
\end{theo}
Note that the result of Alladi-Andrews-Berkovich  uses four primary colors, the full set of
secondary colors, along with one quaternary color $abcd$. When $d=0$, 
we recover \Thm{aag}. Their main tool was a difficult $q$-series
identity:
\begin{align}\label{identity}
\sum_{i,j,k,l-constraints}&
\frac{q^{T_{\tau}+T_{AB}+T_{AC}+T_{AD}+T_{BC}+T_{BD}+T_{CD}-BC-BD-CD+4T_{Q-1}+3Q+2Q\tau}}{
(q)_A(q)_B(q)_C(q)_D(q)_{AB}(q)_{AC}(q)_{AD}(q)_{BC}(q)_{BD}(q)_{CD}(q)_Q}\nonumber\\
\cdot\{&(1-q^A) + q^{A+BC+BD+Q}(1-q^B) + q^{A+BC+BD+Q+B+CD}\}\nonumber\\
=& \quad \frac{q^{T_i+T_j+T_k+T_l}}{(q)_i(q)_j (q)_k(q)_l}
\end{align}
where $A,B,C,D,AB,AC,AD,BC,BD,CD,Q$ are variables which count the number of parts with respectively 
color $a,b,c,d,ab,ac,ad,bc,bd,cd,abcd$, 
\[
\left\lbrace
\begin{array}{l}
i=A+AB+AC+AD+Q\\
j=B+AB+BC+BD+Q\\
k=C+AC+BC+CD+Q\\
l=D+AD+BD+CD+Q\\
\tau = A+B+C+D+AB+AC+AD+BC+BD+CD
\end{array}\right.\,,
\]
$T_n = \frac{n(n+1)}{2}$ is the $n^{th}$ triangular number and $(q)_n = (q;q)_n$. 
While this
identity is difficult to prove, it is relatively straightforward to
show that it is equivalent to the statement in \Thm{th2}.
\\\\In this paper we give a bijective proof of \Thm{th2} (and therefore a bijective proof of the identity \eqref{identity}).   Our
proof is divided into two steps.  First we prove \Thm{th1} below,
which arises more naturally from our methods than \Thm{th2}.
Instead of adding a quaternary color, we lower certain minimum
differences and add some forbidden patterns.  Then, we show how
\Thm{th1} is equivalent to \Thm{th2}.

\subsection{Statement of Results}
Suppose that the parts occur in only primary colors $a,b,c,d$ and secondary colors $ab,ac,ad,bc,bd,cd$, and are ordered as in \eqref{quat} by omitting quaternary parts:
\begin{equation}\label{cons}
1_{ab}< 1_{ac}< 1_{ad} < 1_{a}< 1_{bc} < 1_{bd}< 1_b< 1_{cd}< 1_c<1_d< 2_{ab}<\cdots\,\cdot
\end{equation}
Let us now consider the partitions with the size of the secondary parts greater than one and satisfying the minimal difference conditions in
\begin{equation}\label{tab1}
\begin{array}{|c|cccc|ccc|cc|c|}
\hline
_{\lambda_i}\setminus^{\lambda_{i+1}}&ab&ac&ad&a&bc&bd&b&cd&c&d\\
\hline
ab&2&2&2&2&2&2&2&2&2&2\\
ac&1&2&2&2&2&2&2&2&2&2\\
ad&1&1&2&2&\underline{1}&2&2&2&2&2\\
a&1&1&1&1&2&2&2&2&2&2\\
\hline
bc&1&1&1&1&2&2&2&2&2&2\\
bd&1&1&1&1&1&2&2&2&2&2\\
b&1&1&1&1&1&1&1&2&2&2\\
\hline
cd&\underline{0}&1&1&1&1&1&1&2&2&2\\
c&1&1&1&1&1&1&1&1&1&2\\
\hline
d&1&1&1&1&1&1&1&1&1&1\\
\hline
\end{array}\,,
\end{equation} 
and which avoid the forbidden patterns 
\begin{equation}
((k+2)_{cd},(k+2)_{ab},k_c),((k+2)_{cd},(k+2)_{ab},k_d), ((k+2)_{ad},(k+1)_{bc},k_{a})\,,
\end{equation}
except the pattern $(3_{ad},2_{bc},1_a)$ which is allowed. An example of such a partition is 
\[(11_{ad},10_{bc},8_a,7_{cd},7_{ab},4_c,3_{ad},2_{bc},1_a)\,\cdot\]
We can now state the main theorem of this paper.
\begin{theo}\label{th1}
Let $u,v,w,t,n$ be non-negative integers. Denote by $A(u,v,w,t,n)$ the number of partitions of $n$ into $u$ distncts parts with color $a$, $v$ distinct parts with color $b$, $w$ distinct parts with color $c$ and $t$ distinct parts with color $d$, and denote by $B(u,v,w,t,n)$ the number of partitions of $n$ satisfying the conditions above, with $u$ parts with color $a,ab,ac$ or $ad$, $v$ parts with color $b,ab,bc$ or $bd$, $w$ parts with color $c,ac,bc$ or $cd$ and $t$ parts with color $d,ad,bd$ or $cd$.
We then have $A(u,v,w,t,n)=B(u,v,w,t,n)$, and the corresponding $q$-series identity is given by
\begin{equation}\label{series}
\sum_{u,v,w,t,n\in \N} B(u,v,w,t,n)a^ub^vc^wd^tq^n =(-aq;q)_\infty (-bq;q)_\infty(-cq;q)_\infty(-dq;q)_\infty\,\cdot
\end{equation}
\end{theo}
By
specializing the variables in \Thm{th1}, one can deduce
many partition identities.   For example, by considering the following transformation in \eqref{series}
\begin{equation}\label{dila}
\left\lbrace 
\begin{array}{l r c l }
\text{dilation :} &q &\mapsto&q^{12}\\
\text{translations :} &a,b,c,d &\mapsto&q^{-8},q^{-4},q^{-2},q^{-1}\\
\end{array}
\right. \,,
\end{equation}
we obtain a corollary of \Thm{th1}.
\begin{cor}
For any positive integer $n$, the number of partitions of $n$ into distinct parts congruent to $-2^3,-2^2,-2^1,-2^0\mod 12$ is equal to the number of partitions of $n$ into parts not congruent to $1,5\mod 12$ and different from $2,3,6,7,9$, such that the difference between two consecutive parts is greater than $12$ up to the following exceptions:
\begin{itemize}
\item $\la_i-\la_{i+1}= 9\Longrightarrow \la_i\equiv \pm 3 \mod 12$ and $\la_i-\la_{i+2}\geq 24$,
\item $\la_i-\la_{i+1}= 12\Longrightarrow\la_i\equiv -2^3,-2^2,-2^1,-2^0\mod 12$,
\end{itemize}
except that the pattern $(27,18,4)$ is allowed. 
\end{cor}
For example, with $n=49$, the partitions of the first kind are 
\[(35,10,4),(34,11,4),(28,11,10),(23,22,4),\]
\[(23,16,10),(22,16,11)\,\,\text{and}\,\, (16,11,10,8,4)\]
and the partitions of the second kind are
\[(35,14),(34,15),(33,16),(45,4),(39,10),(38,11)\,\,\text{and}\,\,(27,18,4)\,\cdot\]
\textbf{Corollary 1.1} may be compared
with \textbf{Theorem 3} of \cite{AAB03}, which is \Thm{th2} transformed by \eqref{dila} but
with the dilation $q \mapsto q^{15}$ instead of $q \mapsto q^{12}$.
\\\\The paper is organized as follows. In \Sct{sct2}, we will present some tools that will be useful for the proof of \Thm{th1}. After that, in \Sct{sct3}, we will give the bijection for \Thm{th1}. Then, in \Sct{sct4}, we will prove its well-definedness. Finally, in \Sct{sct5}, we will present and prove the bijection between the partitions with forbidden patterns considered in \Thm{th1} and the partitions with quaternary parts given in \Thm{th2}. In Part II of this series, we will show how our method can be
used to go beyond G\"ollnitz' theorem to any number of primary colors.
\section{Preliminaries}\label{sct2}
\subsection{The setup}
Denote by $\C=\{a,b,c,d\}$ the set of primary colors  and $\Cc=\{ab,ac,ad,bc,bd,cd\}$ the set of secondary colors, and recall the order on $\C\sqcup\Cc$: 
\begin{equation}\label{orD}
ab<ac<ad<a<bc<bd<b<cd<c<d\,\cdot
\end{equation}
We can then define the strict lexicographic order $\succ$ on colored parts by
\begin{equation}\label{lex}
k_p\succ l_{q} \Longleftrightarrow k-l\geq \chi(p\leq q)\,\cdot
\end{equation}
Explicitly, this gives the order 
\begin{equation}\label{cons1}
1_{ab}\prec 1_{ac}\prec 1_{ad} \prec 1_{a}\prec 1_{bc} \prec 1_{bd}\prec 1_b\prec 1_{cd}\prec 1_c\prec 1_d\prec 2_{ab}\prec\cdots\,,
\end{equation}
previously etablished in \eqref{cons}.
We denote by $\Pp$ the set of positive integers with primary color.
\\\\We can easily see that for any $pq\in \Cc$, \textit{with} $p<q$, and any $k\geq 1$, we have that 
\begin{align}
\label{half1}(2k)_{pq} &= k_{q}+k_{p}\\
\label{half2}(2k+1)_{pq} &= (k+1)_{p}+k_{q}\,\cdot
\end{align}
In fact,  any part greater that $1$ with a secondary color $pq$ can be uniquely written as the sum of two consecutive parts in $\Pp$ with colors $p$ and $q$. We then denote by $\Sc$ the set of secondary parts greater than $1$, and define the functions $\alpha$ and $\beta$ on $\Sc$ by 
\begin{equation}\label{ab}
\alpha: \left\lbrace \begin{array}{l c l}
2k_{pq}&\mapsto& k_{q}\\
(2k+1)_{pq}&\mapsto&(k+1)_{p}
\end{array}\right.\qquad \text{and}\qquad\beta: \left\lbrace \begin{array}{l c l}
2k_{pq}&\mapsto& k_{p}\\
(2k+1)_{pq}&\mapsto&k_{q}
\end{array}\right.\,,
\end{equation}
respectively named upper and lower halves. One can check that for any $k_{pq}\in \Sc$, 
\begin{equation}\label{abba}
\alpha((k+1)_{pq}) = \beta(k_{pq})+1\quad\text{and}\quad \beta((k+1)_{pq})=\alpha(k_{pq})\,\cdot
\end{equation}
In the previous sum, adding an integer to a part does not change its color.
We can then deduce by induction that for any $m\geq 0$,
\begin{equation}\label{aj}
\alpha((k+m)_{pq})\preceq \alpha(k_{pq})+m \quad\text{and}\quad \beta((k+m)_{pq})\preceq \beta(k_{pq})+m\,\cdot
\end{equation}
\[\]
Recall the table \eqref{tab2}
\[
\begin{array}{|c|cccc|ccc|cc|c|}
\hline
_{\lambda_i}\setminus^{\lambda_{i+1}}&ab&ac&ad&a&bc&bd&b&cd&c&d\\
\hline
ab&2&2&2&2&2&2&2&2&2&2\\
ac&1&2&2&2&2&2&2&2&2&2\\
ad&1&1&2&2&\textcolor{blue}{2}&2&2&2&2&2\\
a&1&1&1&1&2&2&2&2&2&2\\
\hline
bc&\textcolor{blue}{1}&1&1&1&2&2&2&2&2&2\\
bd&1&1&1&1&1&2&2&2&2&2\\
b&1&1&1&1&1&1&1&2&2&2\\
\hline
cd&1&1&1&1&1&1&1&2&2&2\\
c&1&1&1&1&1&1&1&1&1&2\\
\hline
d&1&1&1&1&1&1&1&1&1&1\\
\hline
\end{array}\,\cdot
\]
It can be viewed as an order $\triangleright$ on $\Pp\sqcup \Sc$ defined by
\begin{equation}\label{Ordre}
k_p\triangleright l_{q}\Longleftrightarrow k-l\geq 1+\left\lbrace
\begin{array}{ll}
\chi(p<q)&\text{if}\quad p\,\,\text{or}\,\,q\in \C\\
\chi(p\leq q)&\text{if}\quad p\,\,\text{and}\,\,q\in \Cc
\end{array}\right.\,\cdot
\end{equation}
By considering the lexicographic order $\succ$, \eqref{Ordre} becomes 
\begin{equation}\label{Ord}
k_p\triangleright l_{q}\Longleftrightarrow\left\lbrace
\begin{array}{ll}
k_p\succeq (l+1)_{q}&\text{if}\quad p\,\,\text{or}\,\,q\in \C\\
k_p\succ (l+1)_{q}&\text{if}\quad p\,\,\text{and}\,\,q\in \Cc
\end{array}\right.\,\cdot
\end{equation}
We can observe that for any primary colors $p,q$
\begin{equation}\label{nog}
k_p\succ l_{q}\quad\text{and}\quad k_p\not\triangleright \,\,l_{q}\quad\Longleftrightarrow \quad k-l=\chi(p<q)\quad\text{and}\quad p\neq q\,,
\end{equation}
and we easily check that in this case, $(k_p,l_q) = (\alpha(k_p+l_{q}),\beta(k_p+l_{q}))$, for $k_p+l_{q}$ viewed as an element of $\Sc$ (see \eqref{half1}, \eqref{half2}).
\\\\We recall that the tables \eqref{tab1} 
\[
 \Delta=\begin{array}{|c|cccc|ccc|cc|c|}
\hline
_{\lambda_i}\setminus^{\lambda_{i+1}}&ab&ac&ad&a&bc&bd&b&cd&c&d\\
\hline
ab&2&2&2&2&2&2&2&2&2&2\\
ac&1&2&2&2&2&2&2&2&2&2\\
ad&1&1&2&2&\textcolor{red}{1}&2&2&2&2&2\\
a&1&1&1&1&2&2&2&2&2&2\\
\hline
bc&1&1&1&1&2&2&2&2&2&2\\
bd&1&1&1&1&1&2&2&2&2&2\\
b&1&1&1&1&1&1&1&2&2&2\\
\hline
cd&\textcolor{red}{0}&1&1&1&1&1&1&2&2&2\\
c&1&1&1&1&1&1&1&1&1&2\\
\hline
d&1&1&1&1&1&1&1&1&1&1\\
\hline
\end{array}\,
\]
and \eqref{tab2} differ only when we have a pair $(p,q)$ of secondary colors such that $(p,q)\in \{(cd,ab),(ad,bc)\}$.  In these cases, the difference in \eqref{tab1} is one less.
\\\\We will now define a relation  $\gg$ on $\Pp\sqcup \Sc$ in such a way that, 
\begin{equation}
k_p\gg l_q \Longleftrightarrow k-l\geq \Delta(p,q)\,\cdot
\end{equation}
Using \eqref{Ord}, this relation can be summarized by the following equivalence :
\begin{equation}\label{Ordd}
k_p\gg l_{q}\Longleftrightarrow\left\lbrace
\begin{array}{ll}
k_p\succeq (l+1)_{q}&\text{if}\quad p\,\,\text{or}\,\,q \in \C\\
k_p\succ (l+1)_{q}&\text{if}\quad p\,\,\text{and}\,\,q \in \Cc\quad\text{and}\quad (p,q)\notin\{(cd,ab),(ad,bc)\}\\
k_p\succ l_{q}&\text{if}\quad (p,q)\in\{(cd,ab),(ad,bc)\}
\end{array}\right.\,\cdot
\end{equation}
We denote by $\Od$ the set of partitions with parts in $\Pp$ and well-ordered by $\succ$. We then have that 
$\la \in \Od$ if and only if there exist $\la_1\succ\cdots\succ \la_t \in \Pp$ such that $\la=(\la_1,\ldots,\la_t)$. We set 
$c(\la_i)$ to be the color of $\la_i$ in $\C$, and $C(\la) = c(\la_1)\cdots c(\la_t)$ as a commutative product of colors in $\C$.
We denote by $\E$ the set of partitions with parts in $\Pp\sqcup \Sc$ and well-ordered by $\gg$. We then have that 
$\nu \in \E$ if and only if there exist $\nu_1\gg\cdots\gg\nu_t \in \Pp\sqcup \Sc$ such that $\nu=(\nu_1,\ldots,\nu_t)$. We set colors $c(\nu_i)\in \C\sqcup\Cc$ depending on whether $\nu_i$ is in  $\Pp$ or $\Sc$, and we also define 
$C(\nu)=c(\nu_1)\cdots c(\nu_t)$ \textit{seen} as a commutative product of colors in $\C$. In fact, a secondary color is just a product of two primary colors. For both kinds of partitions, their size is the sum of their part sizes. 
\\\\We also denote by $\Ee$ the subset of partitions of $\E$ without the forbidden patterns,
\begin{equation}\label{forb}
((k+2)_{cd},(k+2)_{ab},k_c),((k+2)_{cd},(k+2)_{ab},k_d), ((k+2)_{ad},(k+1)_{bc},k_{a})\,,
\end{equation}
except the pattern $(3_{ad},2_{bc},1_a)$ which is allowed. 
We finally define $\Eee$ as the subset of partitions of $\E$ with parts well-ordered by $\triangleright$ in \eqref{Ord}, and we observe that $\Eee$ is indeed a subset of $\Ee$.
\subsection{Technical lemmas}
We will state  and prove some important lemmas for the proof of \Thm{th1}.
\begin{lem}[\textbf{Ordering primary and secondary parts}]\label{lem1}
For any $(l_p,k_{q})\in\Pp\times \Sc$, we have the following equivalences:
\begin{align}
&\quad\l_p\not \gg k_{q}\Longleftrightarrow (k+1)_{q}\gg (l-1)_p\label{oe}\,,\\
&\quad l_p \gg \alpha(k_{q})\Longleftrightarrow \beta((k+1)_{q})\not \succ (l-1)_p\label{eo}\,\cdot
\end{align}
\end{lem}
\begin{lem}[\textbf{Ordering secondary parts}]\label{lem2}
Let us consider the table $\Delta$ in \eqref{tab1}. Then, for any secondary colors $p,q\in \Cc$, 
\begin{equation}\label{gam}
\Delta(p,q) = \min\{k-l:\beta(k_{p})\succ \alpha(l_{q})\}\,\cdot
\end{equation} 
Moreover, if the secondary parts $k_p,l_q$ are such that $\beta(k_p)\succ\beta(\l_q)$, then 
\begin{equation}\label{sw1}
 (k+1)_p\gg l_q\,\cdot
\end{equation}
Furthermore, if $k_{p}\gg l_{q}$, we then have either $\beta(k_{p})\succ \alpha(l_{q})$ or 
\begin{equation}\label{sw}
\alpha(l_{q})+1\gg \alpha((k-1)_{p} )\succ \beta((k-1)_{p})\succ \beta(l_q)\,\cdot
\end{equation}
\end{lem}
\begin{lem}[\textbf{Reversibility $\Od\leftarrow \Ee$}]\label{lem3}
Let us consider a partition $\nu =(\nu_1,\ldots,\nu_t)\in \E$. Then, for any $i\in [1,t-2]$ such that 
$(\nu_{i+1},\nu_{i+2})\in \Sc\times\Pp$ and $(c(\nu_i),c(\nu_{i+1}))\notin \{(ad,bc),(cd,ab)\}$, we have 
\begin{equation}\label{ee}
\nu_{i}\succ \nu_{i+2}+2\,\cdot
\end{equation}
Furthermore, the following are equivalent:
\begin{enumerate}
\item $\nu\in \E_1$,
\item For any $i\in [1,t-2]$ such that $(\nu_i,\nu_{i+1})$ is a pattern in $\{((k+1)_{ad},k_{bc}),(k_{cd},k_{ab})\}$ different from $(3_{ad},2_{bc})$, we have that 
\begin{equation}\label{ee2}
\nu_{i}\succeq \nu_{i+2}+2\,\cdot
\end{equation}
\end{enumerate}
\end{lem}
\begin{proof}[Proof of \Lem{lem1}]
To prove \eqref{oe}, we observe that, for any $(l_p,k_{q})\in\Pp\times \Sc$,  by \eqref{Ordd}, 
\[l_p\not \gg k_{q} \Longleftrightarrow  l_p\not\succeq (k+1)_{q}\,,\]
and
\begin{align*}
(k+1)_{q}\gg (l-1)_p&\Longleftrightarrow (k+1)_{q}\succ l_p\\
 &\Longleftrightarrow (k+1)_{q}\not\preceq l_p\,\cdot
\end{align*}
To prove \eqref{eo}, we first remark that, by \eqref{abba},  $\alpha(k_{q})=\beta((k+1)_{q})$. We then obtain by \eqref{Ordd} that
\[l_p \gg \alpha(k_{q}) \Longleftrightarrow (l-1)_p \succeq \alpha(k_{q})\]
and  
\begin{align*}
\beta((k+1)_{q})\not\succ(l-1)_p&\Longleftrightarrow \alpha(k_{q})\not\succ(l-1)_p\\
&\Longleftrightarrow \alpha(k_{q})\preceq (l-1)_p\,\cdot
\end{align*}
\end{proof}
\begin{proof}[Proof of \Lem{lem2}]
Let us consider $\min\{k-l: \beta(k_{p})\succ \alpha(l_{q})\}$. We just check for the $36$ pairs $(p,q)$ in $\Cc^2$. As an example, we take the  pairs $(cd,ab),(ad,bc)$.
\begin{itemize}
\item For $k=2k'+1$,  we have 
$(\alpha(k_{cd}),\beta(k_{cd}))=((k'+1)_c,k'_d)$. Then, to minimize $k-l$, $\alpha(l_{ab})$ and $\beta(l_{ab})$ have to be the greatest primary parts with color $a,b$ less than $k'_d$. So we obtain $k'_b$ and $k'_a$. We then have 
\[
k-l= 2k'+1-2k'=1\,\cdot
\]
For $k=2k'$,  we have 
$(\alpha(k_{cd}),\beta(k_{cd}))=(k'_d,k'_c)$. Then to minimize $k-l$, $\alpha(l_{ab})$ and $\beta(l_{ab})$ have to be the greatest primary parts with color $a,b$ less than $k'_c$. So we obtain $k'_b$ and $k'_a$. We then have 
\[
k-l= 2k'-2k'=0\,\cdot
\]
Therefore, $\Delta(cd,ab)=0$.
\item We check with the same reasoning by taking for $(ad,bc)$ consecutive parts
\[(k+1)_a\succ k_d\succ k_c\succ k_b\]
and \[(k+1)_d\succ(k+1)_a\succ k_c \succ k_b\,,\]
and we obtain $\Delta(ad,bc)=1$.\\
\end{itemize}
To prove \eqref{sw1}, 
we have by \eqref{abba} that $\alpha((l-1)_q)=\beta(l_q)$. Since $\beta(k_p)\succ \beta(l_q)=\alpha((l-1)_q)$, this then implies by \eqref{gam} that $k_p\gg (l-1)_q$, and this is equivalent to $(k+1)_p\gg l_q$.
\\\\Let us now suppose that $k-l\geq \Delta(p,q)$. We just saw that this minimum value was reached at $k$ or $k-1$. Then if we do not have $\beta(k_{p})\succ \alpha(l_{q})$, we necessarily have $\beta((k-1)_{p})\succ \alpha((l-1)_{q})=\beta(l_{q})$ by \eqref{abba}. Moreover, by \eqref{Ordd}, we have
\[\beta(k_{p})\not\succ \alpha(l_{q})\Longleftrightarrow \alpha(l_{q})+1\gg \alpha((k-1)_{p} )\,,\]
so that we obtain \eqref{sw}.
\end{proof}
\begin{proof}[Proof of \Lem{lem3}]
For any $\nu =(\nu_1,\ldots,\nu_t)\in \E$ and any $i\in [1,t-2]$, we have $\nu_i\gg\nu_{i+1}\gg\nu_{i+2}$.
By \eqref{Ordd}, the fact that
$(\nu_{i+1},\nu_{i+2})\in \Sc\times\Pp$ implies that
\[\nu_{i+1}\succ \nu_{i+2}+1\]
and $(c(\nu_i),c(\nu_{i+1}))\notin \{(ad,bc),(cd,ab)\}$ implies that
\[\nu_i\succ \nu_{i+1}+1 \,,\]
and we thus have \eqref{ee}.
\\\\To prove the second part, we have to show that not having the forbidden patterns in \eqref{forb} is equivalent to the second condition.
\begin{itemize}
\item If we  suppose that $(\nu_i,\nu_{i+1})= (k_{cd},k_{ab})$, we then have by \eqref{Ordd} that
\begin{align*}
\nu_{i+1}\gg \nu_{i+2}&\Longleftrightarrow \nu_{i+1}\succ\nu_{i+2}+1\\
&\Longleftrightarrow k_{ab}\succ\nu_{i+2}+1\\
&\Longleftrightarrow (k-1)_{d}\succeq\nu_{i+2}+1 \quad\text{(by \eqref{cons1})}\,\cdot
\end{align*}
By \eqref{cons1}, we then have that the fact that the patterns $k_{cd},k_{ab},(k-2)_d$ and $k_{cd},k_{ab},(k-2)_c$ are forbidden for $k\geq 3$ is equivalent to $(k-1)_{cd}\succeq\nu_{i+2}+1$, which means that $k_{cd}\succeq \nu_{i+2}+2$.
\item If we suppose that $(\nu_i,\nu_{i+1})= ((k+1)_{ad},k_{bc})$ with $k\geq 3$, we then have by \eqref{Ordd} that
\begin{align*}
\nu_{i+1}\gg \nu_{i+2}&\Longleftrightarrow \nu_{i+1}\succ\nu_{i+2}+1\\
&\Longleftrightarrow k_{cb}\succ\nu_{i+2}+1\\
&\Longleftrightarrow k_{a}\succeq\nu_{i+2}+1\quad\text{(by \eqref{cons1})}\,\cdot
\end{align*}
We then have by \eqref{cons1} that the fact that the pattern $(k+1)_{ad},k_{bc},(k-1)_a$  is forbidden for $k\geq 3$ is equivalent to $k_{ad}\succeq\nu_{i+2}+1$, which means that $(k+1)_{ad}\succeq \nu_{i+2}+2$.
\end{itemize}
\end{proof}
\section{Bressoud's algorithm}\label{sct3}
Here we adapt the algorithm given by Bressoud in his bijective proof of Schur's partition theorem \cite{BR80}. The bijection is easy to
describe and execute, but its justification is more subtle and is
given in the next section.
\subsection{From $\Od$ to $\Ee$} Let us consider the following machine $\Phi$:
\begin{itemize}
\item[\Soo:]For a sequence $\la= \la_1,\ldots,\la_t$, take the smallest $i<t$ such that  $\la_i,\la_{i+1}\in \Pp$ and $\la_i\succ \la_{i+1}$ but $\la_i\not\gg \la_{i+1}$, if it exists, and replace
\begin{equation}
\begin{array}{l c l l}
\la_i &\leftarrowtail& \la_i+\la_{i+1}&\text{as a part in } \Sc\\
\la_j &\leftarrow& \la_{j+1}& \text{for all}\quad i<j<t\,\, 
\end{array}
\end{equation}
and move to \Stt. We call such a pair of parts a \textit{troublesome} pair. We observe that $\la$ loses two parts in $\Pp$ and gains one part in $\Sc$. 
The new sequence is $\la = \la_1,\ldots,\la_{t-1}$. Otherwise, exit from the machine.\\
\item[\Stt:]For $\la= \la_1,\ldots,\la_t$, take the smallest $i<t$ such that $(\la_i,\la_{i+1})\in \Pp\times\Sc$ and $\la_i\not\gg \la_{i+1}$ if it exists,  and replace 
\begin{equation}
(\la_i,\la_{i+1}) \looparrowright (\la_{i+1}+1,\la_i-1)\in \Sc\times\Pp
\end{equation}
and redo \Stt. We say that the parts $\la_i,\la_{i+1}$ are \textit{crossed}. Otherwise, move to \Soo.
\end{itemize} 
Let $\Phi(\la)$ be the resulting sequence after putting  any $\la=(\la_1,\ldots,\la_t)\in \Od$ in $\Phi$. 
This transformation preserves the size and the commutative product of primary colors of partitions.  Let us apply this machine on the partition $(11_c,8_d,6_a,4_d,4_c,4_b,3_a,2_b,2_a,1_d,1_c,1_b,1_a)$.
\[ \begin{array}{ccccccccccccccccc}
\begin{matrix}
11_c\\
8_d\\
6_a\\
\underline{4_d}\\
\underline{4_c}\\
4_b\\
3_a\\
2_b\\
2_a\\
1_d\\
1_c\\
1_b\\
1_a
\end{matrix} &\rightarrowtail&
\begin{matrix}
11_c\\
8_d\\
\mathbf{6_a}\\
\mathbf{8_{cd}}\\
4_b\\
3_a\\
2_b\\
2_a\\
1_d\\
1_c\\
1_b\\
1_a
\end{matrix}&
\looparrowright&
\begin{matrix}
11_c\\
\mathbf{8_d}\\
\mathbf{9_{cd}}\\
5_a\\
4_b\\
3_a\\
2_b\\
2_a\\
1_d\\
1_c\\
1_b\\
1_a
\end{matrix}
&
\looparrowright&
\begin{matrix}
11_c\\
10_{cd}\\
7_d\\
\underline{5_a}\\
\underline{4_b}\\
3_a\\
2_b\\
2_a\\
1_d\\
1_c\\
1_b\\
1_a
\end{matrix}
&
\rightarrowtail&
\begin{matrix}
11_c\\
10_{cd}\\
\mathbf{7_d}\\
\mathbf{9_{ab}}\\
3_a\\
2_b\\
2_a\\
1_d\\
1_c\\
1_b\\
1_a
\end{matrix}
&
\looparrowright&
\begin{matrix}
11_c\\
10_{cd}\\
10_{ab}\\
6_d\\
\underline{3_a}\\
\underline{2_b}\\
2_a\\
1_d\\
1_c\\
1_b\\
1_a
\end{matrix}
&
\rightarrowtail&
\begin{matrix}
11_c\\
10_{cd}\\
10_{ab}\\
6_d\\
5_{ab}\\
\underline{2_a}\\
\underline{1_d}\\
1_c\\
1_b\\
1_a
\end{matrix}
&
\rightarrowtail&
\begin{matrix}
11_c\\
10_{cd}\\
10_{ab}\\
6_d\\
5_{ab}\\
3_{ad}\\
\underline{1_c}\\
\underline{1_b}\\
1_a
\end{matrix}
&
\rightarrowtail&
\begin{matrix}
11_c\\
10_{cd}\\
10_{ab}\\
6_d\\
5_{ab}\\
3_{ad}\\
2_{bc}\\
1_a
\end{matrix}
\end{array}\,\cdot
\]
\subsection{From $\Ee$ to $\Od$}
Let us consider the following machine $\Psi$:
\begin{itemize}
\item[\Soo:]For a sequence $\nu= \nu_1,\ldots,\nu_t$, take the greastest $i\leq t$ such that $\nu_i\in \Sc$ if it exists.  If $\nu_{i+1}\in\Pp$ and $\beta(\nu_i)\not\succ \nu_{i+1}$, then replace 
\begin{equation}
(\nu_i,\nu_{i+1}) \looparrowright (\nu_{i+1}+1,\nu_i-1)\in \Pp\times\Sc
\end{equation}
and redo \Soo. We say that the parts $\nu_i,\nu_{i+1}$ are \textit{crossed}. Otherwise, move to \Stt. If there are no more parts in $\Sc$, exit from the machine.\\
\item[\Stt:]For $\nu= \nu_1,\ldots,\nu_t$, take the the greatest $i\leq t$ such that $\nu_i\in \Sc$. By \Soo, it satisfies $\beta(\nu_i)\succ \nu_{i+1}$. Then replace
\begin{equation}
\begin{array}{l c l l}
\nu_{j+1} &\leftarrow& \nu_{j}& \text{for all}\quad t\geq j>i\,\, \\
(\nu_i) &\rightrightarrows& (\alpha(\nu_i),\beta(\nu_i))&\text{as a pair of parts in }\Pp \,,
\end{array}
\end{equation}
and move to \Soo.  We say that the part $\nu_i$ \textit{splits}. We observe that $\nu$ gains two parts in $\Pp$ and loses one part in $\Sc$. 
The new sequence  is $\nu = \nu_1,\ldots,\nu_{t+1}$.\\
\end{itemize}
Let $\Psi(\nu)$ be the resulting sequence after putting  any $\nu=(\nu_1,\ldots,\nu_t)\in \Ee$ in $\Psi$.
This transformation preserves the size and the product of primary colors of partitions. For example, applying this to $(11_c,10_{cd},10_{ab},6_d,5_{ab},3_{ad},2_{bc},1_a)$ gives
\begin{tiny}
\[ \begin{array}{ccccccccccccccccc}
\begin{matrix}
11_c\\
10_{cd}\\
10_{ab}\\
6_d\\
5_{ab}\\
3_{ad}\\
1_c+1_b\\
1_a
\end{matrix} &
\rightrightarrows&
\begin{matrix}
11_c\\
10_{cd}\\
10_{ab}\\
6_d\\
5_{ab}\\
2_a+1_d\\
1_c\\
1_b\\
1_a
\end{matrix} &
\rightrightarrows&
\begin{matrix}
11_c\\
10_{cd}\\
10_{ab}\\
6_d\\
3_a+2_b\\
2_a\\
1_d\\
1_c\\
1_b\\
1_a
\end{matrix}&
\rightrightarrows&
\begin{matrix}
11_c\\
10_{cd}\\
5_b+5_a\\
6_d\\
3_a\\
2_b\\
2_a\\
1_d\\
1_c\\
1_b\\
1_a
\end{matrix}
&
\looparrowright&
\begin{matrix}
11_c\\
10_{cd}\\
7_d\\
5_a+4_b\\
3_a\\
2_b\\
2_a\\
1_d\\
1_c\\
1_b\\
1_a
\end{matrix}
&
\rightrightarrows&
\begin{matrix}
11_c\\
5_d+5_c\\
7_d\\
5_a\\
4_b\\
3_a\\
2_b\\
2_a\\
1_d\\
1_c\\
1_b\\
1_a
\end{matrix}
&
\looparrowright&
\begin{matrix}
11_c\\
8_d\\
5_c+4_d\\
5_a\\
4_b\\
3_a\\
2_b\\
2_a\\
1_d\\
1_c\\
1_b\\
1_a
\end{matrix}
&
\looparrowright&
\begin{matrix}
11_c\\
8_d\\
6_a\\
4_d+4_c\\
4_b\\
3_a\\
2_b\\
2_a\\
1_d\\
1_c\\
1_b\\
1_a
\end{matrix}&
\rightrightarrows&
\begin{matrix}
11_c\\
8_d\\
6_a\\
4_d\\
4_c\\
4_b\\
3_a\\
2_b\\
2_a\\
1_d\\
1_c\\
1_b\\
1_a
\end{matrix}
\end{array}\,\cdot
\]  
\end{tiny}
\section{Proof of the well-definedness of Bressoud's maps}\label{sct4}
In this section, we will show the following proposition.
\begin{prop}\label{pr}
The transformation $\Phi$ describes a mapping from $\Od$ to $\Ee$ such that $\Psi\circ\Phi = Id_{\Od}$, and  $\Psi$  describes a mapping from $\Ee$ to $\Od$ such that $\Phi\circ\Psi=Id_{\Ee}$.  
\end{prop}
\subsection{Well-definedness of $\Phi$}
In this subsection, we will show the following proposition.
\begin{prop}\label{pr1}
Let us consider any $\la=(\la_1,\ldots,\la_t)\in \Od$, and set $\gamma^0=0$, $\mu^0=\lambda$. Then, in the process $\Phi$ on $\la$, at the $u^{th}$ passage from \St to \Soo, 
there exists a pair of partitions $\gamma^u,\mu^u\in \Ee\times \Od$ such that the sequence obtained is $\gamma^u,\mu^u$. 
Moreover, if we denote by $l(\gamma^u)$ and $g(\mu^u)$ respectively the smallest part of $\gamma^u$ and the greatest part of $\mu^u$, we then have that 
\begin{enumerate}
\item $l(\gamma^u)$ is the $u^{th}$ element in $\Sc$ of $\gamma^u$, 
\item $l(\gamma^u)\gg g(\mu^u)$ so that the partition $(\gamma^u,g(\mu^u))$ is in $\E_1$, 
\item for any $u$, $\gamma^u$ is the beginning of the partition $\gamma^{u+1}$ and the number of parts of $\mu^{u+1}$ is at least two less than the number of parts of $\mu^{u}$.
\end{enumerate}     
\end{prop}
\begin{proof}
Let $\la = (\la_1,\ldots,\la_t)$ be a partition in $\Od$. Let us set $c_1,\ldots,c_t$ to be the primary colors of the parts $\la_1,\ldots,\la_t$.
Now consider the first troublesome pair $\la_i,\la_{i+1}\in \Pp$ obtained at \So in $\Phi$, and the first resulting secondary part $\la_i+\la_{i+1}$. Note that this is reversible by \St of $\Psi$.
\begin{itemize}
\item If there is a part $\la_{i+2}$ after $\la_{i+1}$, we have that
\begin{align*}
\la_i+\la_{i+1}-\la_{i+2}&= \chi(c_i<c_{i+1})+2\la_{i+1}-\la_{i+2}\quad \text{by \eqref{nog}}\\
&\geq \chi(c_i<c_{i+1})+2\chi(c_{i+1}\leq c_{i+2})+\la_{i+2}\quad \text{by \eqref{lex}}\\
&\geq 1+\chi(c_i\leq c_{i+2})+\chi(c_{i+1}\leq c_{i+2})\,\cdot
\end{align*}
Since by \eqref{orD}, we have that $c_i> c_{i+2}$ and $c_{i+1}> c_{i+2}$ implies $c_ic_{i+1}>c_{i+2}$, we then have that
$\la_i+\la_{i+1}-\la_{i+2}\geq 1+\chi(c_ic_{i+1}\leq c_{i+2})$, and we conclude that $\la_i+\la_{i+1}\gg \la_{i+2}$.\\
\item The primary parts to the left of $\la_i$ are well-ordered by $\gg$. We then have
\[\la_1\gg\cdots\gg \la_{i-1}\gg \la_i\,\cdot\]
We obtain by \eqref{Ordd} and  \eqref{aj} that for any $j<i$, 
\begin{align*}
\la_{j}\gg\la_{i}+i-j-1\succeq \alpha(\la_i+\la_{i+1}+i-j-1)
\end{align*}
so that by \eqref{Ordd}, $\la_{j}\gg \alpha(\la_i+\la_{i+1}+i-j-1)$. If after $i-j$ iterations of \Stt, $\la_i+\la_{i+1}$ crosses
$\la_j$, we will have at the same time that 
\begin{align}
(\la_i+\la_{i+1}+i-j)&\gg (\la_j-1)\gg\cdots\gg\la_{i-1}-1\quad\text{(by \eqref{oe})}\label{rg}\\
\beta(\la_i+\la_{i+1}+i-j)&\not\succ (\la_j-1)\gg\cdots\gg\la_{i-1}-1\quad\text{(by \eqref{eo})}\label{reverse1}\,,
\end{align}
and so these iterations are reversible by \So in $\Psi$ (recursively on $j\leq j'<i$).
\item We also have by \eqref{Ordd} that 
\begin{align*}
\la_{i-1}\gg\la_i\succ\la_{i+1}\succ \la_{i+2} &\Longrightarrow \la_{i-1}-1\succeq \la_i\succ\la_{i+1}\succ \la_{i+2}\\
&\Longrightarrow \la_{i-1}-1\succ\la_{i+2}\,\cdot
\end{align*}
\item If we can no longer apply \St after $i-j$ iterations,  we then obtain
\[\la_1\gg\cdots\gg\la_{j-1}\gg(\la_i+\la_{i+1}+i-j)\gg (\la_j-1)\gg\cdots\gg\la_{i-1}-1\succ \la_{i+2}\succ\cdots \succ \la_t\]
and we set
\begin{align}
\gamma^1&=\la_1\gg\cdots\gg(\la_i+\la_{i+1}+i-j)\\
\mu^1&=(\la_j-1)\gg\cdots\gg\la_{i-1}-1\succ \la_{i+2}\succ\cdots\succ\la_t\,,\label{pousse}
\end{align}
and the conditions in the proposition are respected. In fact, even if $j=i$, we saw that $\la_i+\la_{i+1}\gg\la_{i+2}$.\\
\end{itemize}
Now, by applying \So for the second time, we see by \eqref{pousse} that the next troublesome pair is either $\la_{i-1}-1,\la_{i+2}$, or  $\la_{i+2+x},\la_{i+3+x}$ for some $x\geq 0$.
\begin{itemize}
\item If $\la_{i-1}-1\not \gg\la_{i+2}$, this means that $\la_{i-1}-1,\la_{i+2}$ are consecutive for $\succ$, and  \So occurs there. By \eqref{sw1}, we have that 
$(\la_i+\la_{i+1}+1)\gg (\la_{i-1}+\la_{i+2}-1)$.
Then, even if $(\la_{i-1}+\la_{i+2}-1)$ crosses the primary parts $(\la_j-1)\gg\cdots\gg\la_{i-2}-1$ after $i-j-1$ iterations of \Stt, by \eqref{Ordd}, we will still have that 
\[(\la_i+\la_{i+1}+i-j)\gg(\la_{i-1}+\la_{i+2}+i-j-2)\,\cdot\]
\item If $\la_{i-1}-1 \gg\la_{i+2}$, then the next troublesome pair appears at 
$\la_{i+2+x},\la_{i+3+x}$ for some $x\geq 0$, and it forms the secondary part $\la_{i+2+x}+\la_{i+3+x}$.
We also have 
\begin{equation}
\la_i\succ\la_{i+1}\succ\la_{i+2}\gg \cdots \gg \la_{i+2+x}\succ\la_{i+3+x}\,\cdot
\end{equation}
By \eqref{Ordd}, we can easily check that
\[
\la_i\succ\la_{i+1}\succ\la_{i+2} \succeq \la_{i+2+x}+x\succ\la_{i+3+x}+x
\]
so that, by \eqref{gam}, 
\[
(\la_i+\la_{i+1})\gg (\la_{i+2+x}+\la_{i+3+x}+2x)\,\cdot
\]
This means by \eqref{Ordd} that, 
\begin{equation}
(\la_i+\la_{i+1})\gg (\la_{i+2+x}+\la_{i+3+x}+x)
\end{equation}
and, as soon as $x\geq 1$, by \eqref{Ordre}
\begin{equation}\label{must}
(\la_i+\la_{i+1})\triangleright (\la_{i+2+x}+\la_{i+3+x}+x)\, \cdot
\end{equation}
We then obtain that, even if the secondary part $\la_{i+2+x}+\la_{i+3+x}$ crosses,  after $x+i-j$ iterations of \Stt, the primary parts
\[\la_j-1\gg\cdots\gg (\la_{i-1}-1)\gg \la_{i+2}\gg \cdots \gg \la_{i+1+x}\,,\]
we will still have
\[(\la_i+\la_{i+1}+i-j)\gg(\la_{i+2+x}+\la_{i+3+x}+x+i-j)\,\cdot\]
\end{itemize}
However, as soon as $x\geq 1$, we directly have 
\[(\la_i+\la_{i+1}+i-j)\triangleright(\la_{i+2+x}+\la_{i+3+x}+x+i-j)\,\cdot\]
In that case, the pair $(\la_i+\la_{i+1}+i-j,\la_{i+2+x}+\la_{i+3+x}+x+i-j)$ cannot have the form $(k_{cd},k_{ab})$ or $((k+1)_{ad},k_{bc})$. In order to have these patterns, we must necessarily have that the second 
troublesome pair is  either $(\la_{i-1}-1,\la_{i+2})$ or $(\la_{i+2},\la_{i+3})$.
In both cases, we can see that either both parts crossed the primary part $l_s$ to the right of the pattern,  or they do not move backward, so that the lower half of the second secondary part is greater than the primary part $l_s$ to the right of the pattern in terms of $\succ$. In the first case, we have that
\begin{align*}
&(l+2)_s \not \gg\la_i+\la_{i+1}+i-j-1\\
\Longleftrightarrow\quad&(l+2)_s \not \succeq \la_i+\la_{i+1}+i-j \quad \text{by \eqref{Ordd}}
\end{align*}
and then $\la_i+\la_{i+1}+i-j\succ (l+2)_s$, so that the forbidden patterns in \eqref{forb} do not occur. 
In the second case, we check the different subcases:
\begin{align*}
(2k'_{cd},2k'_{ab},l_s)&\Longrightarrow k'_a\succ l_s\\
&\Longrightarrow k'-l\geq \chi(a\leq s)=1\\
&\Longrightarrow 2k'-l\geq l+2\geq 3 \\
&\Longrightarrow 2k'-(l+2)\geq 1\geq \chi(cd \leq s)\,,
\end{align*}
\begin{align*}
((2k'+1)_{cd},(2k'+1)_{ab},l_s)&\Longrightarrow k'_b\succ l_s\\
&\Longrightarrow k'-l\geq \chi(b\leq s)\\
&\Longrightarrow 2k'+1-l\geq l+1+2\chi(b\leq s)\\
&\Longrightarrow 2k'+1-l\geq 2+2\chi(b\leq s)\\
&\Longrightarrow (2k'+1)-(l+2)\geq \chi(b\leq s)\geq \chi(cd\leq s)\quad\text{since}\quad b<cd\,\cdot 
\end{align*}
\begin{align*}
((2k'+2)_{ad},(2k'+1)_{bc},l_s)&\Longrightarrow k'_c\succ l_s\\
&\Longrightarrow k'-l\geq \chi(c\leq s)\\
&\Longrightarrow 2k'-l\geq l+2\chi(c\leq s)\\
&\Longrightarrow 2k'-l\geq 1+2\chi(c\leq s)\\
&\Longrightarrow (2k'+2)-(l+2)\geq 1+\chi(c\leq s)\geq \chi(ad\leq s)\,\cdot
\end{align*}
\begin{align*}
((2k'+1)_{ad},2k'_{bc},l_s)&\Longrightarrow k'_b\succ l_s\\
&\Longrightarrow k'-l\geq \chi(b\leq s)\\
&\Longrightarrow 2k'+1-l\geq l+1+2\chi(b\leq s)\\
&\Longrightarrow 2k'+1-(l+2)\geq l-1+\chi(b\leq s)\geq 1-\chi(l=1)\chi(s=a)\\
\end{align*}
We can see that only $(3_{ad},2_{bc},1_a)$ does not satisfy the fact that $(l+2)_s$ is less than the first part $\la_i+\la_{i+1}+i-j$. Recall that the second part needs to be greater than $(l+1)_s$.
By \Lem{lem3}, we then have the forbidden patterns 
\[((k+2)_{cd},(k+2)_{ab},k_c),((k+2)_{cd},(k+2)_{ab},k_d), ((k+2)_{ad},(k+1)_{bc},k_{a})\,,\]
with only $(3_{ad},2_{bc},1_a)$ allowed. The conditions in the proposition are satisfied after the second move from \St to \Soo.
\\\\By induction, \Prp{pr1} follows. Moreover, by \eqref{reverse1}, every single step is reversible by $\Psi$, since by its application 
the sequence $\gamma^{u+1},\mu^{u+1}$ becomes exactly after the iterations of \So and splitting in \St the sequence $\gamma^u,\mu^u$ (with the last part of $\gamma^{u+1}$). 
\end{proof}
The fact that $\Phi(\Od)\subset \E_1$ follows from \Prp{pr1} since $\mu^u$ strictly decreases in terms of number of parts and the process stops as soon as either $\mu^u$ has at most one part, or all its primary parts are well-ordered by $\gg$. And the reversibility implies that $\Psi\circ \Phi_{|\Od }= Id_{\Od}$.
\subsection{Well-definedness of $\Psi$}
In this subsection, we will show the following proposition.
\begin{prop}\label{pr2}
Let us consider any $\nu=\nu_1,\ldots,\nu_t\in \Ee$, and set $\gamma^0=\nu$, $\mu^0=0$. Then, in the process $\Psi$ on $\nu$, at the $u^{th}$ passage  from \St to \Soo, there exists a pair of partitions $\gamma^u,\mu^u\in \Ee\times \Od$ such that the sequence obtained is $\gamma^u,\mu^u$. Moreover, if we denote by $l(\gamma^u)$ and $g(\mu^u)$ respectively the smallest part of $\gamma^u$ and the greatest part of $\mu^u$, we then have that 
\begin{enumerate}
\item $l(\gamma^u)\in \Pp$, 
\item $l(\gamma^u)$and $g(\mu^u)$ are consecutive for $\succ$,
\item for any $u$, $\mu^u$ is the tail of the partition $\mu^{u+1}$ and the number of secondary parts of $\gamma^u$ decreases by one at each step. 
\end{enumerate}    
\end{prop}
\begin{proof}
If the pattern $(3_{ad},2_{bc},1_a)$ is in $\nu$, these parts are then the last ones. By applying $\Psi$, we obtain after the second passage at the tail of the partition the sequence $2_a,1_d,1_c,1_b,1_a$. Now suppose that this pattern does not occur in $\nu$.
Let us consider the last secondary part $\nu_i$ of $\nu$. 
\begin{itemize}
\item Suppose that \So does not occur and we directly have \Stt. If there is a part $\nu_{i-1}$ to its left, and $(\nu_{i-1},\nu_i) \notin \{((k+1)_{ad},k_{bc}),(k_{cd},k_{ab})\}$, we then have $\nu_{i-1}\triangleright \nu_i$ and 
\begin{align*}
\nu_{i-1}-\alpha(\nu_i)&= \nu_{i-1}-\nu_i+\beta(\nu_i)\\
&\geq 2\quad\quad (\text{by \eqref{Ordre} and the fact that}\,\,\beta(\nu_i)\geq 1)\,,
\end{align*}
so that $\nu_{i-1}\gg\alpha(\nu_i)$. In the case that $(\nu_{i-1},\nu_i) \in \{((k+1)_{ad},k_{bc}),(k_{cd},k_{ab})\}$, a quick check according to the parity of $k$ shows that we also have $\nu_{i-1}\gg\alpha(\nu_i)$.
If we have the pattern $(\nu_{i-2},\nu_{i-1})\in\{((k+1)_{ad},k_{bc}),(k_{cd},k_{ab})\}$, then $\nu_i\preceq \nu_{i-2}-2$, and
\begin{align*}
\nu_{i-2}-\alpha(\nu_i)&= \nu_{i-2}-\nu_i+\beta(\nu_i)\\
&\geq 3\quad\quad \text{by \eqref{Ordre} and the fact that}\,\,\beta(\nu_i)\geq 1\,\,
\end{align*}
so that $\nu_{i-2}\succ(\alpha(\nu_i)+2)$. Note that by \Lem{lem3}, this implies that $\nu_{i-2},\nu_{i-1},\alpha(\nu_i)$ cannot be a forbidden pattern.\\
 
\item If $\nu_i$ crosses after iteration of \So the primary parts $\nu_{i+1}\gg \cdots \gg \nu_{j}$, we then have 
\begin{align}
\nu_{i-1}&\gg\nu_{i+1}+1\gg\cdots \gg \nu_{j}+1\gg \alpha(\nu_i-j+i) \quad\text{(by \eqref{eo})}\label{orddd}\\
\nu_{i-1}&\gg\nu_{i+1}+1\gg\cdots \gg \nu_{j}+1\not\gg (\nu_i-j+i) \label{reverse2}\,\cdot
\end{align}
In fact, by \eqref{ee}, if  $(\nu_{i-1},\nu_i) \notin \{((k+1)_{ad},k_{bc}),(k_{cd},k_{ab})\}$, we necessarily have that $\nu_{i-1}\gg\nu_i\gg\nu_{i+1}$ so that $\nu_{i-1}\gg\nu_{i+1}+1$. 
If $(\nu_{i-1},\nu_i) \in \{((k+1)_{ad},k_{bc}),(k_{cd},k_{ab})\}$, since $\nu_{i-1}\succ(\nu_{i+1}+2)\in \Pp$, we necessarily have by \eqref{Ordd} that $\nu_{i-1}\gg(\nu_{i+1}+1)$. 
\\If we have the pattern $(\nu_{i-2},\nu_{i-1})\in\{((k+1)_{ad},k_{bc}),(k_{cd},k_{ab})\}$, then $\nu_i\preceq \nu_{i-2}-2$, and  
\[
\nu_{i-2}\succeq \nu_i+2\succ \nu_{i+1}+3\,\cdot
\]
So $\nu_{i-2}\succ\nu_{i+1}+3$, and the pattern $\nu_{i-2},\nu_{i-1},\nu_{i+1}+1$ is not forbidden.
\\Finally, since $\nu_i\gg \nu_{i+1}\gg\cdots\gg \nu_j$ and $\nu_{i+1},\ldots,\nu_j\in \Pp$, we then have by \eqref{Ordd} that
$\nu_i\gg \nu_j+j-i-1$, 
and this is equivalent by \eqref{oe} to $\nu_{j}+1\not\gg (\nu_i-j+i)$. This implies that all these iterations of \So are reversible by \St of $\Phi$.\\
\item In the case $j=t$, we have by \eqref{reverse2}, \eqref{oe} and \eqref{Ordd} that 
\[\nu_i-t+i\succ \nu_t\,\cdot \]
If we suppose that $\nu_i-t+i$ has size $1$, then $\nu_t$ has also size $1$ and a color smaller than the color of $\nu_i$. But by \eqref{half1} and \eqref{orD}, we necessarily have that $\beta(\nu_i-t+i+1)$ has size $1$ and a color greater than 
the color of $\nu_i$. We then obtain by \eqref{cons1} that 
\[\beta(\nu_i-t+i+1)\succ \nu_i-t+i\succ \nu_t\,,\]
so that we do not cross $\nu_i-t+i+1$ and $\nu_t$. This is absurd by assumption. 
In any case, after crossing, we still have that the secondary part size is greater than $1$, so that after splitting, its upper and lower halves stay in $\Pp$.\\ 
\item If we stop the iteration of \So just before $\nu_{j+1}$, this means by \eqref{orddd}  that
\begin{equation}
\nu_{i-1}\gg \nu_{i+1}+1\gg\cdots \gg \nu_{j}+1\gg\alpha(\nu_i-j+i)\succ \beta(\nu_i-j+i)\succ\nu_{j+1}\gg\cdots\gg\nu_t\,\cdot
\end{equation}
We then set 
\begin{align*}
\gamma^1 &= \nu_1\gg\cdots\gg\nu_{i-1}\gg \nu_{i+1}+1\gg \cdots \gg\nu_{j}+1\gg\alpha(\nu_i-j+i)\,,\\
\mu^1&=\beta(\nu_i-j+i)\succ\nu_{j+1}\succ\cdots\succ\nu_t\,,
\end{align*}
and we saw with all the different cases that the conditions of \Prp{pr2} are respected.\\
\end{itemize}
Let us now consider the secondary part $\nu_{i-x}$ before $\nu_i$, for some $x\geq 1$. Then, by iteration of \Soo, it can never
cross $\beta(\nu_i-j+i)$. In fact, suppose that it crosses all primary parts
$\nu_{i-x+1}\gg\cdots \gg\nu_{i-1} \gg \nu_{i+1}+1\gg\cdots \gg \nu_{j}+1$. We then obtain 
$\nu_{i-x}-x+1+i-j$, and since 
\[\nu_{i-x}\gg\nu_{i-x+1}\gg\cdots\gg\nu_{i-1}\gg\nu_i,\]
and $\nu_{i-x+1},\ldots,\nu_{i-1}\in \Pp$, we have by \eqref{Ordd} that 
 $\nu_{i-x}-x+1\gg \nu_i$, which is equivalent to $\nu_{i-x}-x+1+i-j\gg \nu_i-j+i$. We obtain by \eqref{sw} that
\begin{align*}
&\text{either}\quad\beta(\nu_{i-x}-x+1+i-j)\succ\alpha(\nu_i-j+i)\\
&\text{or}\quad\alpha(\nu_i-j+i)+1\gg\alpha(\nu_{i-x}-x+i-j)\succ\beta(\nu_{i-x}-x+i-j)\succ \beta(\nu_i-j+i)\,\cdot
\end{align*}
In any case, the splitting in \St occurs before $\beta(\nu_i-j+i)$. We set then 
\begin{align*}
\gamma^2 &= \nu_1\gg\cdots\gg\alpha(\nu_{i-x}-y)\,,\\
\mu^2&=\beta(\nu_{i-x}-y)\succ\cdots\succ \beta(\nu_i-j+i)\succ\nu_{j+1}\succ\cdots\succ\nu_t\,,
\end{align*}
where $y$ is the number of iterations of \So before moving to \Stt, 
and by reasoning as before for the different cases,  we can easily see that the conditions of \Prp{pr2} are respected. We obtain the result recursively. We also observe that the sequence $\gamma^u,\mu^u$ is exactly what we obtain by applying successively iteration of \So and \St of the transformation $\Phi$ on $\gamma^{u+1},\mu^{u+1}$.
\end{proof}
By the lemma, since the number of secondary parts decreases by one at each passage from \St to \Soo, we will stop after exactly the number of secondary parts in $\nu$. And the result is of the form 
$\gamma^U,\mu^U\in \Od$ with  $\gamma^U$ well-ordered by $\gg$, and the last part of the first partition and the first of the second partition are consecutive in terms of $\succ$. We then conclude that $\Psi(\Ee)\subset \Od$. Since all the steps are reversible by $\Phi$, we also have $\Phi\circ\Psi_{|\Ee}=Id_{\Ee}$. 

\section{Bijective proof of \Thm{th2}}\label{sct5}

In this section, we will describe a bijection for \Thm{th2}. For brevity, we refer to the partitions in \Thm{th2} as quaternary partitions.
\subsection{From $\Ee$ to quaternary partitions}

We consider the patterns $((k+1)_{ad},k_{bc}),(k_{cd},k_{ab})$ and we sum them as follows :
\begin{align}
(k+1)_{ad}+k_{bc}&= (2k+1)_{abcd}\nonumber\\
k_{cd}+k_{ab}&= 2k_{abcd}\,\label{decomp}\,\cdot
\end{align}
\paragraph{Let us now take a partition $\nu$ in $\Ee$}
We then identify all the patterns $(M^i,m^i)\in \{((k+1)_{ad},k_{bc}),(k_{cd},k_{ab})\}$ and suppose that 
\[\nu = \nu_1,\ldots,\nu_x,M^1,m^1,\nu_{x+1},\ldots,\nu_y,M^2,m^2,\nu_{y+1},\ldots, M^t,m^t,\ldots,\nu_s\,\cdot\]
\textit{As long as} we have a pattern $\nu_j,M^i,m^i$, we cross the parts by replacing them using 
\begin{equation}
\nu_j,M^i,m^i \longmapsto M^i+1,m^i+1,\nu_j-2\,\cdot
\end{equation}
At the end of the process, we obtain a final sequence 
\[N^1,n^1,N^2,n^2,\ldots,N^t,n^t,\nu'_1,\ldots,\nu'_s\,\cdot\]
Finally, the associated pair of partitions is set to be $(K^1,\ldots,K^t),\nu'=(\nu'_1,\ldots,\nu'_t)$, where 
$K^i=N^i+n^i$ according to  \eqref{decomp}.
\\To sum up the previous transformation, we only remark that, for each quaternary part $K^i$ obtained by summing of the original pattern $M^i,m^i$, we add twice the number of the remaining primary and secondary parts in $\nu$ to the left of the pattern that gave $K^i$, 
while we subtract from these parts two times the number of quaternary parts obtained by patterns that occur to their right.
\\\\With the example $11_c,\underline{10_{cd},10_{ab}},6_d,5_{ab},\underline{3_{ad},2_{bc}},1_a$,
\[
\begin{footnotesize}
\begin{array}{ccccccccccccc}\begin{matrix}
 11_c\\
 10_{cd}\\
 10_{ab}\\
 6_d\\
 5_{ab}\\
 3_{ad}\\
  2_{bc}\\
 1_a
\end{matrix}
&\mapsto&
\begin{matrix}
 11_c\\
 10_{cd},10_{ab}\\
 6_d\\
 5_{ab}\\
 3_{ad}, 2_{bc}\\
 1_a
\end{matrix}
&\mapsto&
\begin{matrix}
 11_{cd},11_{ab}\\
 9_c\\
 6_d\\
 5_{ab}\\
 3_{ad},2_{bc}\\
 1_a
\end{matrix}
&\mapsto&
\begin{matrix}
 11_{cd},11_{ab}\\
 9_c\\
 6_d\\
 4_{ad},3_{bc}\\
 3_{ab}\\
 1_a
\end{matrix}
&\mapsto&
\begin{matrix}
 11_{cd},11_{ab}\\
 9_c\\
 5_{ad},4_{bc}\\
 4_d\\
 3_{ab}\\
 1_a
\end{matrix}
&\mapsto&
\begin{matrix}
 11_{cd},11_{ab}\\
6_{ad},5_{bc}\\
 7_c\\
 4_d\\
 3_{ab}\\
 1_a
\end{matrix}
\end{array}\,\cdot
\end{footnotesize}
\]
 we obtain $[(22_{abcd},11_{abcd}),(7_c,4_d,3_{ab},1_a)]$. We now proceed to show that the
image of this mapping is indeed a quaternary partition.    The inverse
mapping will be presented in the next subsection.
\begin{enumerate}
\item \textbf{Quaternary parts are well-ordered}. Let us consider two consecutive patterns $(M^j,m^j)=(k_p,l_q)$ and $(M^{j+1},m^{j+1})=(k'_{p'},l'_{q'})$. Since $\nu$ is well-ordered by $\gg$, we have by \eqref{Ordd} and \eqref{Ord} that 
\begin{equation}
l_q\triangleright \l^1_{p_1}\triangleright\cdots\triangleright l^i_{p_i}\triangleright k'_{p'}\,\cdot
\end{equation}
By \eqref{Ord}, we then have that $l_q\succ k'_{p'}+i+1$ so that $l-k'\geq i+1+\chi(q\leq p')$.
 Since by \eqref{Ordd}, $k-l=\chi(p\leq q)$ and $k'-l'=\chi(p'\leq q')$, we then have  that 
 \begin{align*}
 k+l-(k'+l')&= \chi(p\leq q)+\chi(p'\leq q')+2(l-k')\\
            &\geq \chi(p\leq q)+\chi(p'\leq q')+2\chi(q\leq p')+2i+2
 \end{align*}
 and we obtain that 
 \begin{align*}
 \chi(cd\leq ab)+&\chi(cd\leq ab)+2\chi(ab\leq cd)= 2\\
 \chi(cd\leq ab)+&\chi(ad\leq bc)+2\chi(ab\leq ad)= 3\\
 \chi(ad\leq bc)+&\chi(cd\leq ab)+2\chi(bc\leq cd)= 3\\
 \chi(ad\leq bc)+&\chi(ad\leq bc)+2\chi(bc\leq ad)= 2\,,
 \end{align*}
 so that  $k+l-(k'+l')\geq 4+2i$. We will then have, after adding twice the remaining primary and secondary elements to their left, that the difference between two consecutive 
 quaternary parts will be at least $4$.\\
\item \textbf{The partition $\nu'$ is in $\Eee$}. Let us consider two consecutive elements $\nu_x=k_p, \nu_{x+1}=l_q$. We then have for consecutive patterns $M^u,m^u$ between $k_p$ and $l_q$  that 
\begin{equation}
k_p\triangleright M^i\gg m^i \gg \cdots \gg M^j\gg m^j\triangleright l_q\,\cdot
\end{equation}
For the case $(M^j, m^j, l_q) \neq (3_{ad},2_{bc},1_a)$, 
since by \Lem{lem3}, $M^{u}\succeq M^{u+1}+2$, $M^j\succeq l_q+2$, and by \eqref{Ordd}, we have that $k_p \succ M^i+1$, and then
\begin{equation}
k_p\succ 1+2(j-i+1)+l_q \Longrightarrow k_p \triangleright 2(j-i+1)+l_q\,\cdot
\end{equation}
For the case $(M^j, m^j, l_q) =(3_{ad},2_{bc},1_a)$, we obtain that 
\begin{equation}
k_p-2(j-i+1)+1\succ 3_{ad}
\end{equation}
and this means that $k_p-2(j-i+1)+1\succeq 3_a $ so that $k_p-2(j-i+1)\succeq 2_a\triangleright 1_a$ .
\\\\In any case, $k_p \triangleright 2(j-i+1)+l_p$, and this implies that  after the subtraction of twice the number of the quaternary parts obtained to their right, these parts will be well-ordered by $\triangleright$.\\
\item \textbf{The minimal quaternary part is well-bounded}. 
Let us first suppose that the tail of $\nu$ consists only of patterns $M^u,m^u$. We then have that
\[\nu_s\triangleright M^i\gg m^i \gg \cdots \gg M^t\gg m^t\]
and, by \Lem{lem3} and \eqref{Ordd}, $\nu_s-2(t-i+1)+1\succeq M^t\succeq 2_{cd}$, so that $\nu'_s=\nu_s-2(t-i+1)\succeq 1_{cd} \succ 1_a$.  
This means that $1_a\notin \nu'$. We also obtain that $K^t = M^t+m^t+2s \geq 2s+4$.
\\\\Now suppose that the tail of $\nu$ has the form
\begin{equation}
l_q\triangleright \nu_u\triangleright\cdots\triangleright \nu_s\,,
\end{equation}
with $M^t,m^t = k_p,l_q$. By \eqref{Ord}, we obtain that $l_q\succ \nu_s+s-u+1$.
\begin{itemize}
 \item If $\nu_s=1_a$, we then have 
\begin{align*}
k+l &= \chi(p\leq q)+2l\\
&\geq  \chi(p\leq q)+ 2(s-u+2+\chi(q\leq a))\\
&= 2(s-u+1)+ 2+\chi(p\leq q)+2\chi(q\leq a)\,,
\end{align*}
and with $(p,q)\in \{(ad,bc),(cd,ab)\}$ we have 
\begin{align*}
\chi(ad\leq bc )&+2\chi(bc\leq a) = 1\\
\chi(cd\leq ab)&+2\chi(ab\leq a)=2
\end{align*}
so that $k+l \geq 2(s-u+1)+ 3$. Then after the addition of $2(u-1)$ for the remaining primary and secondary parts of $\nu$ to the left of the pattern $(M^t,m^t)$, we obtain that the smallest quaternary part is at least $2s+3$. Note that $\nu'_s=\nu_s=1_a$.\\
\item When $\nu_s=h_r\neq 1_a$, we obtain that 
\begin{align*}
k+l&\geq  \chi(p\leq q)+ 2(s-u+1+h+\chi(q\leq r))\\
&= 2(s-u+1)+ 2h+\chi(p\leq q)+2\chi(q\leq r)\,,
\end{align*}
so that if $h\geq 2$, then $k+l\geq 2(s-u+1)+4$. If not, $h=1$, and since there is no secondary part of size $1$, we necessary have that $r\geq b$, so that $\chi(q\leq r)=1$ whenever $q\in\{ab,bc\}$. 
We thus obtain $k+l\geq 2(s-u+1)+4$. We then conclude that for $\nu_s\neq 1_a$, the smallest quaternary part is
at least $2s+4$.
\end{itemize}
In any case, we have that the smallest quaternary part is at least $2s+4-\chi(1_a\in \nu')$.
\end{enumerate}
\subsection{From quaternary partitions to $\Ee$}
Recall  by \eqref{decomp} that $K_{abcd}$ splits as follows :
\begin{align*}
(k+1)_{ad}+k_{bc}&= (2k+1)_{abcd}\\
k_{cd}+k_{ab}&= 2k_{abcd}\,
\end{align*}
Let  us then consider partitions 
$(K^1,\ldots,K^t)$ and $\nu=(\nu_1,\ldots,\nu_s)\in \Eee$, with quaternary part $K^u$ such that $K^t\geq 4+2s-\chi(1_a\in \nu)$ and $K^u-K^{u+1}\geq 4$.
We also set $K^u = (k^u,l^u)$ the decomposotion according to \eqref{decomp}. We then proceed as follows by beginning with $K^t$ and $\nu_1$, 
\begin{itemize}
\item[\Soo:] If we do not encounter $K^{u+1}=(k^{u+1},l^{u+1})$ and $\nu_i\neq 1_a$ and $\nu_i+2\triangleright k^u-1$, then replace
\begin{align*}
\nu_i&\longmapsto \nu_i+2\\
(k^u,l^u)&\longmapsto (k^u-1,l^u-1)
\end{align*}
and move to $i+1$ and redo \Soo. Otherwise, move to \Stt.
\item[\Stt] If we encounter $K^{u+1}= k^{u+1}\gg l^{u+1}$, then split $(k^u,l^u)$ into $k^u\gg l^u$.
If not, it means that we have met $\nu_i$ such that $\nu_i+2\not\triangleright\,\,\, k^u-1$. Then we split $k^u\gg l^u$. Since we have $\nu_i+2\not\triangleright \,\,k^u-1$, which is equivalent by \eqref{Ord} to $ k^u\succeq\nu_i+2$, by \Lem{lem3}, this is exactly the condition to avoid the forbidden patterns, with $k^u\gg l^u\triangleright \nu_i$.
\\We can now move to \So with $u-1$ and $i=1$.
\end{itemize}
With the example $[(22_{abcd},11_{abcd}),(7_c,4_d,3_{ab},1_a)]$, we obtain
\[
\begin{footnotesize}
\begin{array}{ccccccccccccccccc}
\begin{matrix}
 11_{cd},11_{ab}\\
6_{ad},5_{bc}\\
 7_c\\
 4_d\\
 3_{ab}\\
 1_a
\end{matrix}
&\mapsto&
\begin{matrix}
 11_{cd},11_{ab}\\
 9_c\\
 5_{ad},4_{bc}\\
 4_d\\
 3_{ab}\\
 1_a
\end{matrix}
&\mapsto&
\begin{matrix}
 11_{cd},11_{ab}\\
 9_c\\
 6_d\\
 4_{ad},3_{bc}\\
 3_{ab}\\
 1_a
\end{matrix}
&\mapsto&
\begin{matrix}
 11_{cd},11_{ab}\\
 9_c\\
 6_d\\
 5_{ab}\\
 3_{ad},2_{bc}\\
 1_a
\end{matrix}
&\mapsto&
\begin{matrix}
 11_{cd},11_{ab}\\
 9_c\\
 6_d\\
 5_{ab}\\
 3_{ad}\\
 2_{bc}\\
 1_a
\end{matrix}
&\mapsto&
\begin{matrix}
 11_c\\
 10_{cd},10_{ab}\\
 6_d\\
 5_{ab}\\
 3_{ad}\\
 2_{bc}\\
 1_a
\end{matrix}
&\mapsto&
\begin{matrix}
 11_c\\
 10_{cd}\\
 10_{ab}\\
 6_d\\
 5_{ab}\\
 3_{ad}\\
 2_{bc}\\
 1_a
\end{matrix}
\end{array}\,\cdot
\end{footnotesize}
\]
It is easy to check that when two quaternary parts meet in \Stt, we will always have $l^u\gg k^{u+1}$, since this is exactly the condition for the minimal difference $K^u-K^{u+1}\geq 4$ and they crossed the same number of $\nu_i$. We can also check that even if the minimal part crossed $\nu_1,\ldots,\nu_s\neq 1_a$, we will still have at the end $K^t\geq 4$ and for $\nu_s=1_a$, $K^t\geq 5$. We see with \eqref{decomp} that 
the size of $m^t$ is at least equal to $2$, and for the case $\nu_s=1_a$, $m^t$ is at least equal to $2_{bc}\gg 1_a$. The partition obtained is then in $\Ee$. 

\end{document}